\documentclass[11pt]{article}

\usepackage{amsthm,amsmath,amssymb}
\usepackage{graphicx,color,epsfig}
\usepackage[margin=1in]{geometry}

\def\N{\mathbb{N}}
\def\R{\mathbb{R}}

\DeclareMathOperator{\var}{Var}
\DeclareMathOperator{\RV}{RV}

\newtheorem{Thm}{Theorem}
\newtheorem{Pro}[Thm]{Proposition}
\newtheorem{Lem}[Thm]{Lemma}

\begin{document}

\title{The variation of invariant graphs in forced systems}
\author{Bastien Fernandez$^1$ and Anthony Quas$^2$}
\date{}   
\maketitle
\begin{center}
$^1$ Laboratoire de Probabilit\'es, Statistique et Mod\'elisation\\
CNRS - Universit\'e Paris 7 Denis Diderot - Sorbonne Univ.\\
75205 Paris CEDEX 13 France\\

$^2$ Department of Mathematics and Statistics\\
University of Victoria\\
Victoria, BC, V8W 3R4, Canada
\end{center}

\begin{abstract}
In skew-product systems with contractive factors, all orbits asymptotically approach the graph of the so-called sync function; hence, the corresponding regularity properties primarily matter. In the literature, sync function Lipschitz continuity and differentiability have been proved to hold depending on the derivative of the base reciprocal, if not on its Lyapunov exponent. However, forcing topological features can also impact the sync function regularity. Here, we estimate the total variation of sync functions generated by one-dimensional Markov maps. A sharp condition for bounded variation is obtained depending on parameters, that involves the Markov map topological entropy. The results are illustrated with examples. 
\end{abstract}

\leftline{\small\today.}
\bigskip

{\bf To describe the properties of the long term response to a deterministic stimulus is a ubiquitous issue in Chaotic Dynamics. In the context of dissipative factors driven by autonomous systems, this question boils down to evaluating the characteristics of the so-called synchrony function. So far, focus has been made on sync function features that depend on the forcing derivative (in a broad sense). However, discontinuous examples in applications suggest the need to evaluate other basic features, such as the total variation. Here, we prove in simple examples that the sync function is, or is not, of bounded variation, depending only on the forcing topological entropy. Therefore, basic properties of invariant graphs also depend on topological features of the corresponding forcing systems.}
\bigskip

\section{Introduction}
An important issue in the theory of forced systems is to evaluate the regularity of their synchronization graph. Suppose that an autonomous (discrete time) dynamical system
\[
\text{\tt x}^{t+1}=f(\text{\tt x}^t),\ t\in\N\cup\{0\}
\]
is given together with a dissipative factor
\[
\text{\tt y}^{t+1}=g(\text{\tt x}^t,\text{\tt y}^t).
\]
Here $f$ is an invertible map with invariant subset $U$ in a Banach space $X$ with norm $\|\cdot\|_X$ and $g:X\times Y\to Y$ ($Y$ is a Banach space with norm $\|\cdot\|_Y$) is such that
\[
\sup_{\text{\tt x}\in U}\|g(\text{\tt x},\text{\tt y})-g(\text{\tt x},\text{\tt y}')\|_Y\leq \gamma \|\text{\tt y}-\text{\tt y}'\|_Y,\ \forall \text{\tt y},\text{\tt y}'\in Y
\]
for some $0<\gamma<1$. Then the orbits $\{(\text{\tt x}^t,\text{\tt y}^t)\}$ of the skew-product system $(f,g)$ are attracted by the graph of the corresponding sync function \cite{AVR86,KP96,PC90,RSTA96}. This sync function, say $\phi:U\to Y$, can be defined by the conjugacy equation
\[
g(\text{\tt x},\phi(\text{\tt x}))=\phi\circ f(\text{\tt x}),\ \forall \text{\tt x}\in U,
\]
which ensures invariance of the corresponding graph $\{(\text{\tt x},\phi(\text{\tt x}))\}_{\text{\tt x}\in U}$.
In this context, regularity properties of this function matter because they determine those dynamical characteristics of the drive system that carry over to the factor. For instance, Lipschitz continuity implies control of dimension estimates. Applications range from filtering of chaotic signals \cite{BBDRCPR88,BHM92,C10} to damage detection in material science \cite{NTSM07}.

The study of sync functions can be regarded as part of the analysis of inertial manifolds in dynamical systems \cite{HP70,HPS77}. Beside existence and continuity, a standard result in this theory is the proof of differentiability under the condition 
\[
K\gamma<1
\]
where $K=\sup\limits_{\text{\tt x}\in U}\|Df^{-1}(\text{\tt x})\|_X$, $f$ is a diffeomorphism and $g$ is continuously differentiable \cite{CD96,HOY97,S97,S99}. (If $f^{-1}$ and $g$ are merely Lipschitz continuous then the same inequality - where $K$ now stands for the Lipschitz constant of $f^{-1}$ - implies that the function $\phi$ itself is Lipschitz continuous.) When $K\gamma\geq 1$, $\phi$ may not be differentiable/Lipschitz continuous \cite{BP87,KM-PY84}; however it is certainly H\"older continuous \cite{ACC01,CD96,S97,S99,U04}.

The analysis has subsequently been extended to accommodate non-uniformly hyperbolic effects. In particular, $\phi$ has been proved to be differentiable in the Whitney sense ({\sl i.e.}\ given an invariant measure $\mu$ for $f$, $\phi$ is uniformly differentiable on sets of measure $1-\epsilon$ for every $\epsilon>0$) provided that \cite{CD96,S99}
\[
e^{\text{Lyap}(f^{-1})}\gamma<1
\]
where $\text{Lyap}(f^{-1})$ is the maximum of the largest Lyapunov exponent of $f^{-1}$ over a set of full measure $\mu$, and satisfies $e^{\text{Lyap}(f^{-1})}\leq K$ .  

Investigations have been pursued beyond the homeomorphic case, either when $f$ is non-invertible \cite{BJMSS03,RA03,U04}, or when it has discontinuities, more precisely, when $f^{-1}$ is a discontinuous map of the interval \cite{ACC01,BP87,BHM92}. In this case, discontinuities transfer to the sync function, where they typically form dense subsets. A natural quantifier in this setting is the total variation which, roughly speaking, measures the length of the corresponding graph, see e.g.\ \cite{KF99} for a definition. The sync function has been shown to be of bounded variation under the condition
\[
e^{h_\text{top}(f^{-1})}\gamma<1
\]
where $h_\text{top}(f^{-1})$ is the topological entropy of $f^{-1}$. On the other hand, the total variation of $\phi$ becomes infinite when $\gamma$ is sufficiently large \cite{F12}. A sync function of bounded variation makes it easier to show that the factor statistics inherits absolute continuity of the forcing statistics \cite{FQ11}. In any case, topological feature of $f^{-1}$ can also affect regularity properties of $\phi$.

\section{Main results}
This letter aims to provide similar total variation estimates in more general cases, when the forcing is not necessarily discontinuous. For simplicity, we shall work in the framework of (one-dimensional) linear filters and skew-product inverse forcing (although certain results extend to more general cases without additional conceptual difficulties) \cite{BBDRCPR88,BHM92,DC96,U04}. More precisely, the set $U$ can be written as $U=[0,1]\times U_2$ (where $U_2$ will be irrelevant), and letting $\text{\tt x}=(x,\text{\tt x}_2)$ where $x\in [0,1]$, the inverse forcing is given by
\[
f^{-1}(x,\text{\tt x}_2)=(T(x),f_2(x,\text{\tt x}_2))
\]
where $T:[0,1]\to [0,1]$ (and, again, the mapping $f_2$ will play no role here). For a systematic way of defining $f$ such that $f^{-1}$ is as here, see \cite{F12}. Assuming in addition that the filter satisfies
\[
g(\text{\tt x},\text{\tt y})=x+\gamma \text{\tt y},\ \forall \text{\tt y}\in\R=Y, \text{\tt x}=(x,\text{\tt x}_2)\in[0,1]\times U_2,
\]  
the sync function then takes the following compact expression
\begin{equation}
\phi_{\gamma}=\sum\limits_{k=0}^{+\infty}\gamma^kT^{k+1}.
\label{DEFPHIT}
\end{equation}
A Markov map of the interval \cite{PY98} is said to be transitive if the corresponding transition matrix is irreducible. The main result of this paper states that, for $T$ a piecewise affine expanding and transitive Markov map, 
the total variation of $\phi_\gamma$ shows a sharp transition at $\gamma=e^{-h_\text{top}(T)}$. 
\begin{Thm}
Assume that $T$ is a piecewise affine and expanding, and transitive Markov map. Then, the total variation of $\phi_\gamma$ defined by \eqref{DEFPHIT} on $[0,1]$
\begin{itemize}
\item[$\bullet$] is finite when $e^{h_\text{top}(T)}\gamma<1$, and
\item[$\bullet$] is infinite when $e^{h_\text{top}(T)}\gamma>1$, except maybe for at most $2N-1$ values of $\gamma$, where $N$ is the number of atoms of $T$.
\end{itemize}
\label{MAINRES}
\end{Thm}
\begin{figure}[ht]
\begin{center}
\includegraphics*[width=150mm]{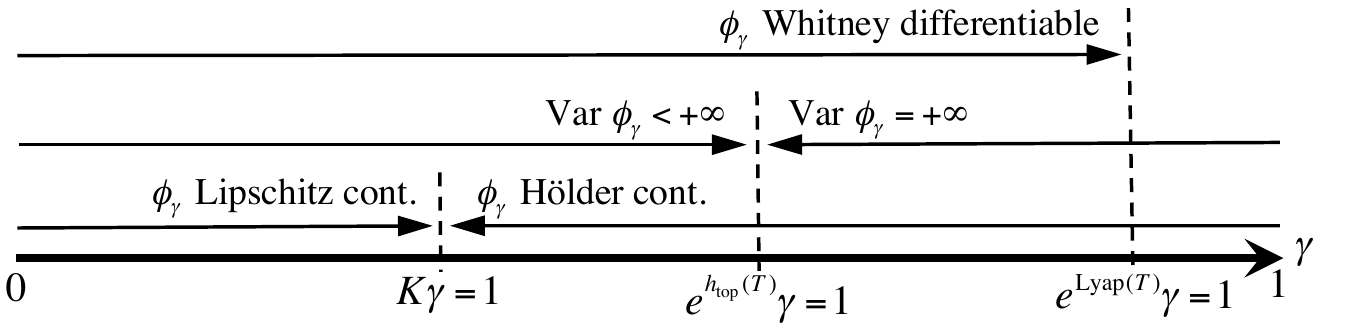}
\end{center}
\caption{Schematic diagram of the properties of $\phi_{\gamma}={\displaystyle\sum\limits_{k=0}^{+\infty}}\gamma^kT^{k+1}$ for $0<\gamma<1$, when $T$ is a continuous piecewise affine and expanding, and transitive Markov map of the interval.}
\label{SCHEMDIAG}
\end{figure}
For maps $T$ that are also continuous, this result can be combined with the previously mentioned ones to obtain the schematic diagram of the sync function properties presented in Fig.\ \ref{SCHEMDIAG}.

As the proof will show, the first part of the statement actually holds for every piecewise continuous and monotone map $T$ of the interval, and also for more general filters such as 
\[
g(\text{\tt x},\text{\tt y})=p(x)+\gamma \text{\tt y}
\]
where $p$ is assumed to be smooth. An example of such extension is $p(x)=\cos(2\pi x)$ and $T(x)=\beta x-\lfloor \beta x\rfloor$, where the sync function essentially boils down to the Weierstrass function \cite{KM-PY84} (NB: In this case we have $K=e^{h_\text{top}(T)}=e^{\text{Lyap}(T)}=\beta$). 

On the other hand, the proof of the second part strongly relies on the assumptions on $T$, although we believe that infinite variation occurs for sync functions associated with more general maps. Furthermore, we also believe that $\phi_\gamma$ has infinite variation for all $\gamma\geq e^{-h_\text{top}(T)}$, i.e.\ there are no exceptional values. This can be proved in some special cases (see examples in Fig.\ \ref{FIGALLLAMB}).
\begin{figure}[ht]
\begin{center}
\includegraphics*[height=48mm]{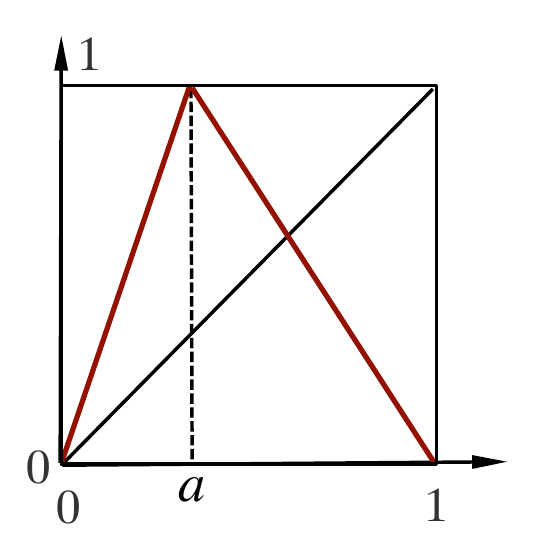}
\hspace{0.75cm}
\includegraphics*[height=48mm]{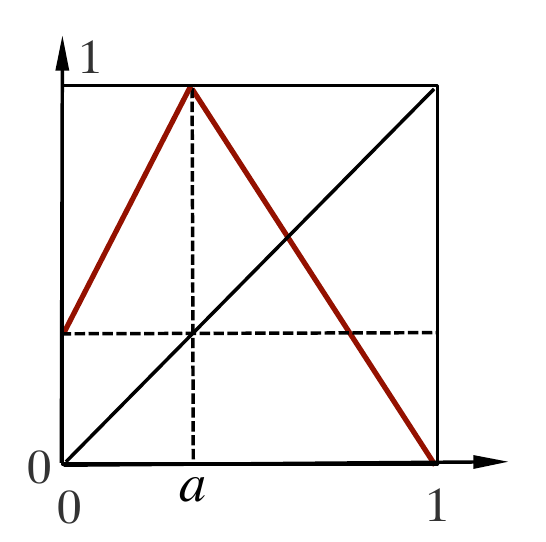}
\hspace{0.75cm}
\includegraphics*[height=48mm]{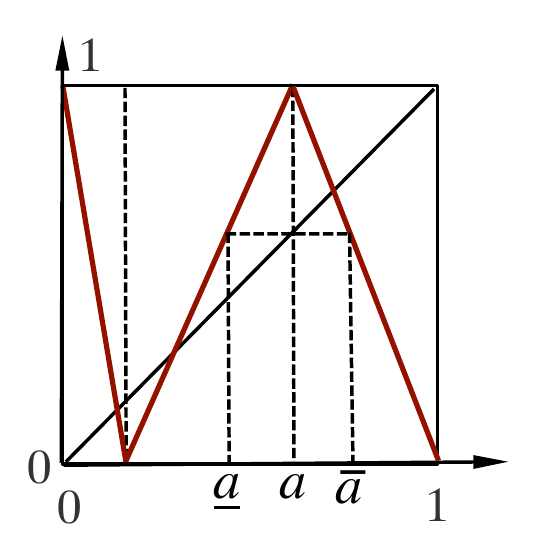}
\end{center}
\caption{Examples of maps satisfying the assumptions of Proposition \ref{ALLLAMB}.}
\label{FIGALLLAMB}
\end{figure}
\begin{Pro}
Assume that $T$ is a piecewise affine and expanding, and transitive Markov map with a full branch and such that
\begin{itemize}
\item either $T(0+0)=0$ or $T(1-0)=1$,
\item or there exist $\underline{a}<a<\overline{a}$ such that $T(\underline{a})=T(\overline{a})=a$ and either $T(a)=0$ and $T(0)>0$ or $T(a)=1$ and $T(1)<1$.
\end{itemize}
Then, the total variation of $\phi_\gamma$ on $[0,1]$ is infinite when $e^{h_\text{top}(T)}\gamma\geq 1$.
\label{ALLLAMB}
\end{Pro}
To illustrate the results, examples of graphs of the sync function $\phi_\gamma$, obtained with the tent map (left picture in Fig.\ \ref{FIGALLLAMB}) are presented in Fig.\ \ref{EXAMPL}.
\begin{figure}[ht]
\begin{center}
\includegraphics*[height=45mm]{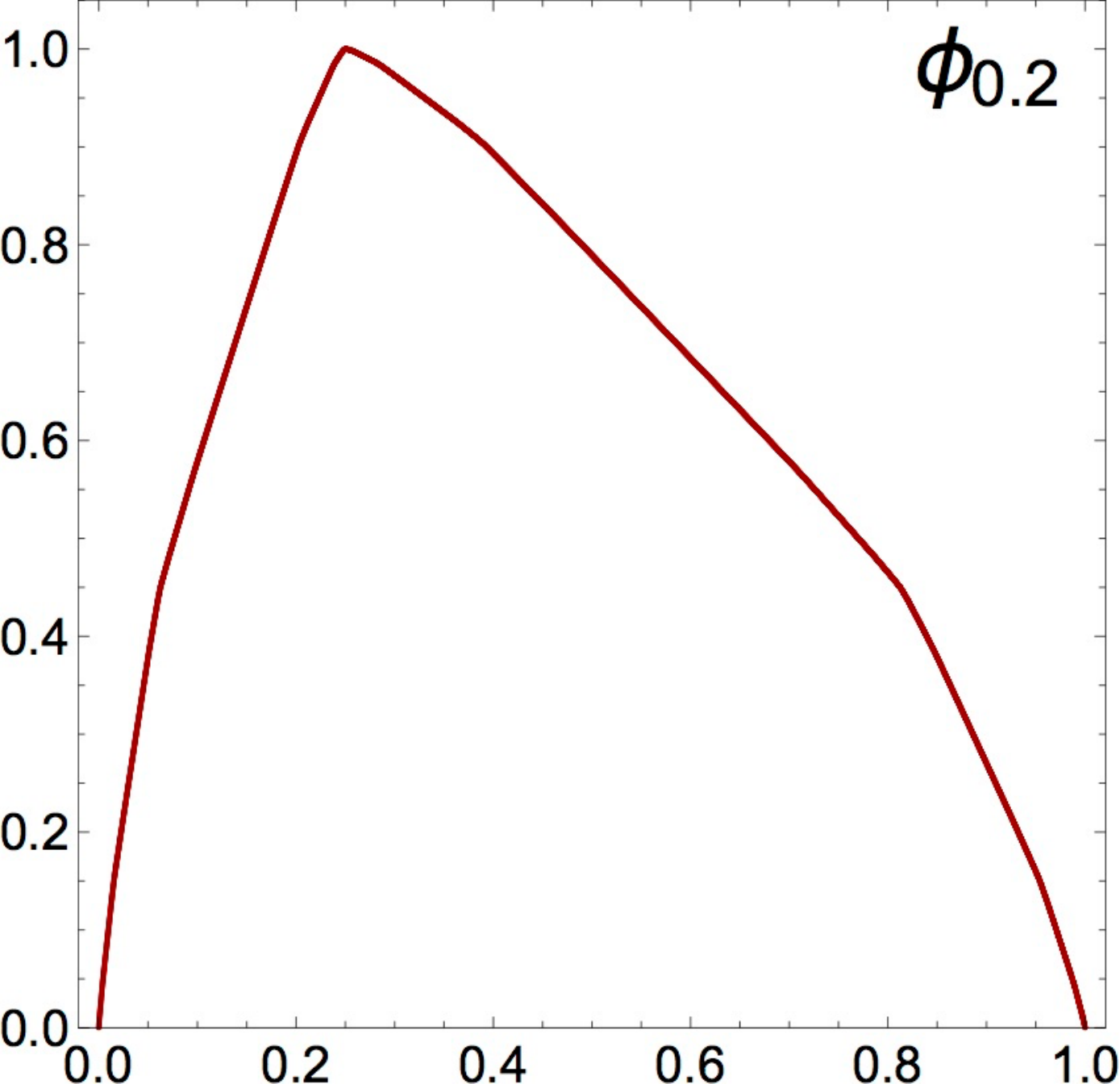}
\hspace{0.75cm}
\includegraphics*[height=45mm]{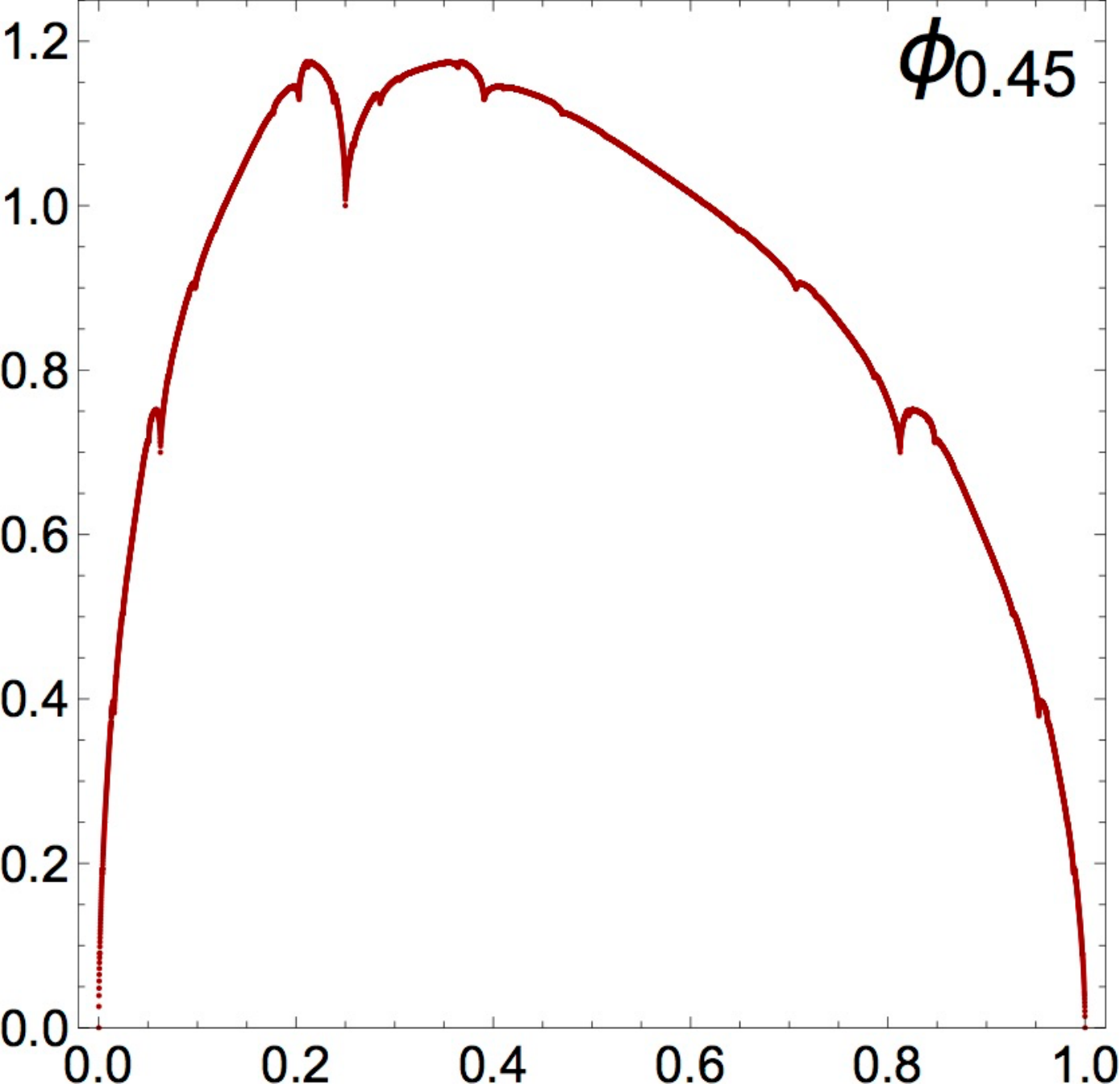}
\end{center}
\begin{center}
\includegraphics*[height=45mm]{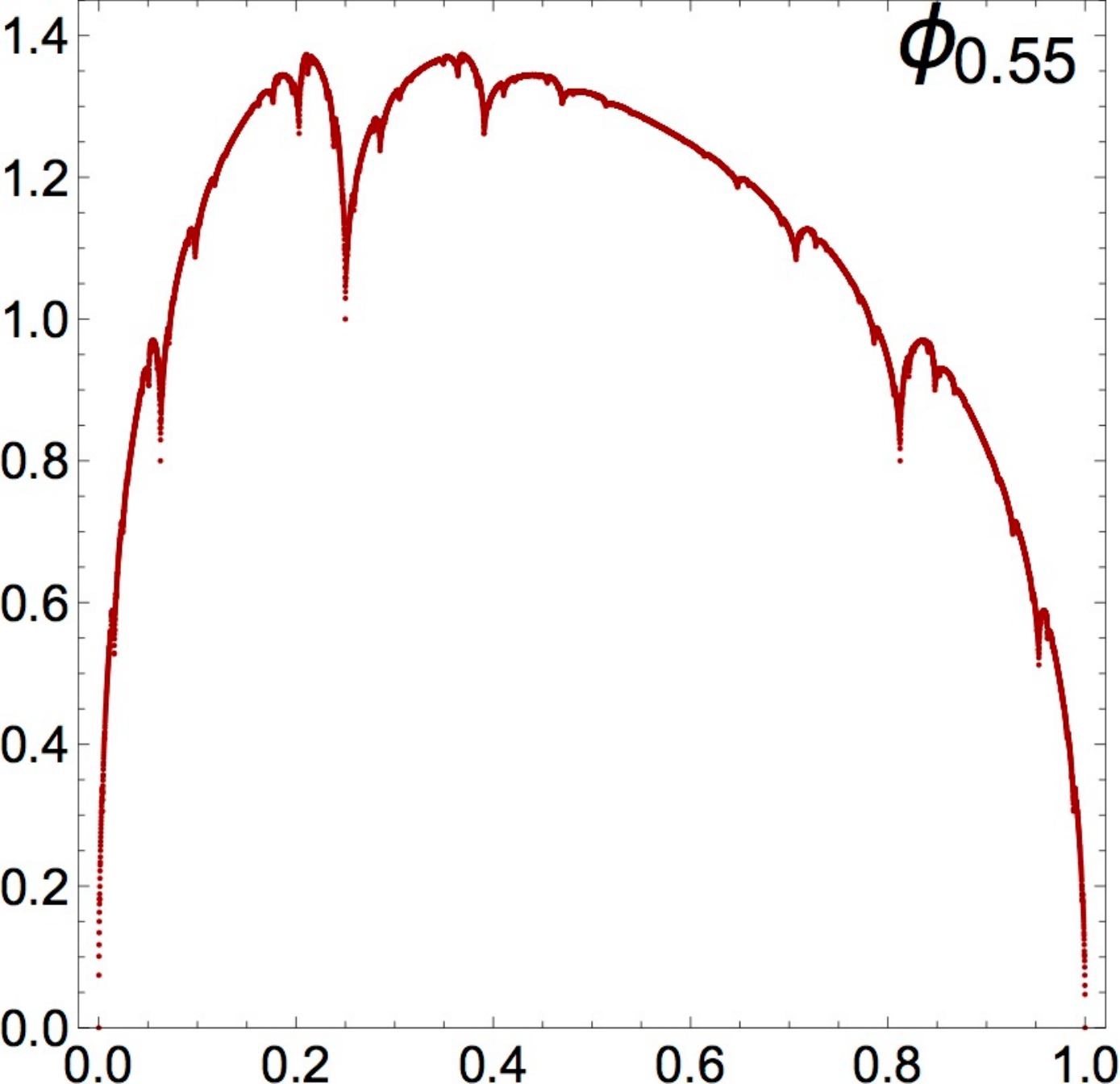}
\hspace{0.75cm}
\includegraphics*[height=45mm]{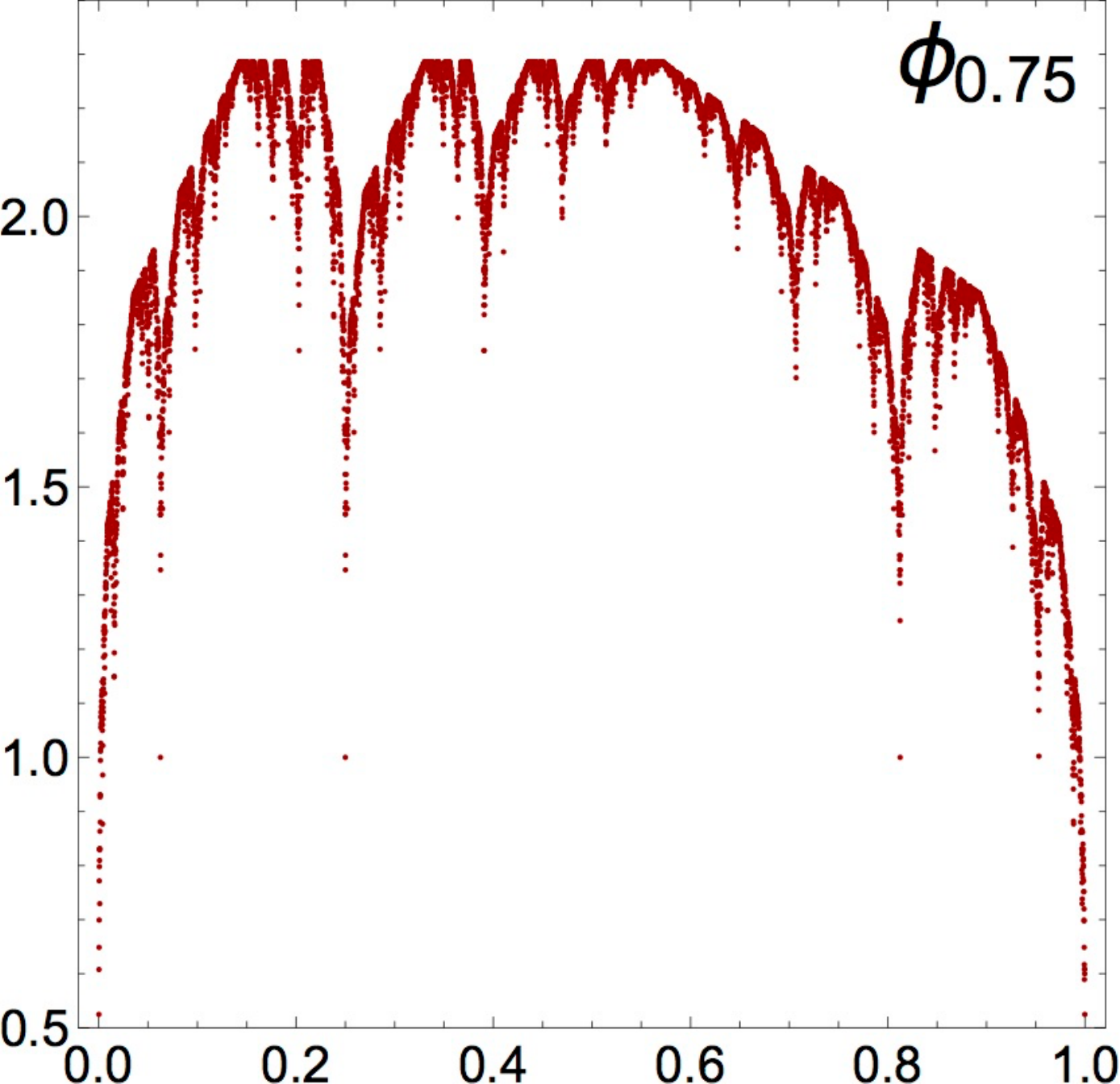}
\end{center}
\caption{Examples of the function $\phi_\gamma$ for the tent map with maximum at $x=a$ (left picture in Fig.\ \ref{FIGALLLAMB}) and $a=\tfrac14$ here. Top left: $\gamma=0.2$ is such that $K\gamma <1$. Top right: $\gamma= 0.45$ (viz.\ $K\gamma<1<e^{h_\text{top}(T)}\gamma$). Bottom left: $\gamma=0.55$ ($e^{h_\text{top}(T)}\gamma<1<e^{\text{Lyap(T)}}\gamma$). Bottom right: $\gamma=0.75$ ($e^{\text{Lyap(T)}}\gamma>1$).}
\label{EXAMPL}
\end{figure}

\section{Proofs}
{\sl Proof of Theorem \ref{MAINRES}.} In order to prove the first assertion, consider an arbitrary subdivision $0=x_1<x_2<\cdots<x_{n}=1$ of $[0,1]$. Uniform convergence in the definition of $\phi_\gamma$ implies 
\[
\sum_{i=1}^{n-1}\left|\phi_\gamma(x_{i+1})-\phi_\gamma(x_{i})\right|\leq \sum_{k=0}^{+\infty}\gamma^k\sum_{i=1}^{n-1}\left|T^{k+1}(x_{i+1})-T^{k+1}(x_{i})\right|.
\]
Using that ${\displaystyle \sum_{i=1}^{n-1}}\left|T^{k+1}(x_{i+1})-T^{k+1}(x_{i})\right|\leq \text{Var}_{[0,1]}T^{k+1}$ and taking the supremum over all subdivisions then yields
\[
\text{Var}_{[0,1]}\phi_\gamma\leq \sum_{k=0}^{+\infty}\gamma^k\text{Var}_{[0,1]}T^{k+1}.
\]
Now, given $k\geq 0$, let $n_k$ be the number of affine branches of $T^k$. Since each branch has variation at most 1 and the jump between branches is at most 1, we have $\text{Var}_{[0,1]}T^{k}\leq 2n_k$ and the first assertion follows from the fact that for a Markov map, we have $h_\text{top}(T)= {\displaystyle \lim_{k\to +\infty}\frac{\log n_k}{k}}$ \cite{PY98}. 

For the second assertion, let $\{I_i\}_{i=1}^N$ be the intervals on which $T$ is affine and let $s_i$ be the slope on $I_i$. Let $X$ be the subshift of finite type consisting of admissible sequences of elements of $\mathcal A=\{1,2,\dots,N\}$ and let $M$ be its transition matrix. 

Given a real function $f$ defined on an interval $I$, define the \emph{reduced variation} of $f$ on $I$ by 
$$
\RV_I(f)=\inf_a\var_I(f(x)-ax).
$$
The reduced variation satisfies the following properties:
\begin{itemize}
\item $\RV_I(f)=\RV_I(f+g)$ for any affine function $g$ on $I$;
\item $\RV_I(\alpha f)=\alpha \RV_I(f)$ for every $\alpha\in\R^+$;
\item $\RV_{[a,c]}(f)\ge \RV_{[a,b]}(f)+\RV_{[b,c]}(f)$, for every $a<b<c$;
\item $\RV_I(f)=0$ if and only if $f$ is affine on $I$.
\end{itemize}
\begin{Lem}\label{lem:affine}
Assume that $T$ is a piecewise affine and transitive Markov map and let $\gamma> e^{-h_\text{top}(T)}$. If $\text{\rm Var}_{[0,1]}\phi_\gamma<+\infty$, 
then $\phi_\gamma$ must be affine on each $I_i$. 
\end{Lem}
\begin{proof}
Let $v_i=\RV_{I_i}(\phi_\gamma)$. The relation $\phi_\gamma=T+\gamma\phi_\gamma\circ T$, the above properties, the Markov property and the fact that the branches are linear imply 
\[
v_i=\RV_{I_i}(T+\gamma\phi_\gamma\circ T)=\gamma\RV_{I_i}(\phi_\gamma\circ T)\geq \gamma\sum_{i\to j}\RV_{I_{i}\cap T^{-1}(I_j)}(\phi_\gamma\circ T)=\gamma \sum_{i\to j}v_j.
\]
That is, $v\ge \gamma Mv$ and $v=\{v_i\}_{i=1}^N$ is a non-negative vector. Since order is preserved, we
can iterate this inequality giving $v\ge (\gamma M)^nv$ for all $n\in\N$. The assumption $\gamma> e^{-h_\text{top}(T)}$ implies that $\gamma A$
is a non-negative irreducible matrix with spectral radius greater than 1. According to the Perron-Frobenius Theorem, the only 
non-negative $v$ satisfying $v\ge \gamma Mv$ is $v=0$, viz.\ $\phi_\gamma$ must be affine on each $I_i$.
\end{proof}

It remains to show that there are at most finitely many values of $\gamma$ such that $\phi_\gamma$ is affine on each interval $I_i$. To that goal, we separate the proof into two cases:
\begin{itemize}
\item[(1)] there exist $j\in \{1,\dots,N\}$, $x,y\in I_j$ and $n\in\N$ such that $T'(T^n(x))\neq T'(T^n(y))$,
\item[(2)] there exists $n\in\N$ such that $T^n$ contains a discontinuity inside some $I_i$.
\end{itemize}
One of these conditions must always hold. Indeed, if all iterates $T^n$ have constant derivatives on each $I_i$, then some iterate $T^n$ must be discontinuous inside some $I_i$ (otherwise, $T$ being piecewise expanding, the length $|T^n(I_i)|$ of a sufficiently large iterate would be larger than 1, which is impossible).

\noindent
{\em Case (1).} We may wlog assume that $n$ is the smallest integer such that $T'(T^n(x))\neq T'(T^n(y))$. Since $T$ is piecewise affine, there must exist $z\in (x,y)\cap T^{-n}(a)$ where $a$ is an atom boundary point for which the left and right derivative $T'(a-0)\neq T'(a+0)$ differ. The Markov assumption implies that the orbits of boundary points must be eventually periodic. Moreover, we may also assume that $T$ is continuous at all iterates of $a$; otherwise case (2) applies. Therefore, there exist $p,q\geq 0$ such that $p+q\leq N+1$ and an atom boundary point $b$ so that 
\[
T^{q}(a)=b\quad\text{and}\quad T^{p}(b)=b.
\]
Applying repeatedly the relation $\phi_\gamma=T+\gamma\phi_\gamma\circ T$ to successive iterates, one gets
\[
\phi_\gamma(z)=\sum_{k=0}^{n-2}\gamma^kT^{k+1}(z)+\gamma^{n-1}\left(\sum_{k=0}^{q-1}\gamma^kT^{k}(a)+\frac{\gamma^{q}}{1-\gamma^{p}}\sum_{k=0}^{p-1}\gamma^kT^{k}(b)\right).
\]
All iterates $T^k$ involved in the RHS are piecewise affine; hence the left and right derivatives $(T^k)'(\cdot-0$ and $(T^k)'(\cdot+0)$ are well-defined at every point in the interior of $(0,1)$. These derivatives can be extended by continuity to $[0,1]$ (viz.\ $T'(0-0):=T'(0+0)$ and $T'(1+0):=T'(1-0)$), so that $\phi_\gamma$ certainly has well-defined left and right derivatives at $z$.

That $n$ is the smallest integer implies $(T^k)'(z-0)=(T^k)'(z+0)$ for $k=1,\dots n-1$. It follows that 
the equation $(\phi_\gamma)'(z-0)=(\phi_\gamma)'(z+0)$ amounts to a polynomial one, with degree at most $p+q-1\leq N$, and which is non-degenerate since $T'(a-0)\neq T'(a+0)$. Therefore, except for at most $N$ values of $\gamma$, we have $(\phi_\gamma)'(z-0)\neq (\phi_\gamma)'(z+0)$ and then $\phi_\gamma$ cannot be affine on $I_j$, as desired. 

\noindent
{\em Case (2).} The argument is similar to the previous case. Let $x$ in the interior of some $I_i$ and a discontinuity $a$ of $T$ be such that $T^n(x)=a$ for some $n\in\N$. Wlog, we can assume that $T$ is continuous at each $T^k(x)$ for $k=1,\dots,n-1$. Then we have 
\[
\left|\phi_\gamma(x-0)-\phi_\gamma(x+0)\right|=\gamma^n\left|\phi_\gamma(a-0)-\phi_\gamma(a+0)\right|.
\]
Let $a_-=T(a-0)$ and $a_+=T(a+0)$. Again, the orbits of $a_-,a_+$ must be eventually periodic. More precisely, there exist $q_-,p_-\geq 0$, $q_-+p_-\leq 2N$, $\{a_{-,k}\}_{k=1}^{q_--1}$ and $\{b_{-,k}\}_{k=0}^{p_--1}$ such that
\[
\phi_\gamma(a-0)=a_-+\sum_{k=1}^{q_--1}a_{-,k}\gamma^k+\frac{\gamma^{q_-}}{1-\gamma^{p_-}}\sum_{k=0}^{p_--1}b_{-,k}\gamma^k
\]
and similarly
\[
\phi_\gamma(a+0)=a_++\sum_{k=1}^{q_+-1}a_{+,k}\gamma^k+\frac{\gamma^{q_+}}{1-\gamma^{p_+}}\sum_{k=0}^{p_+-1}b_{+,k}\gamma^k.
\]
Similarly to before, the equation $\phi_\gamma(a-0)= \phi_\gamma(a+0)$ amounts to a non-degenerate polynomial one, since $a_-\neq a_+$. The polynomial degree is at most $\max\{q_-,q_+\}+p_-+p_+-1\leq 2N-1$ if $p_-\neq p_+$ (resp.\ $\max\{q_-,q_+\}+p+-1\leq 2N-1$ if $p_-=p_+=p$). Hence, except for at most $2N-1$ values of $\gamma$, $\phi_\gamma$ cannot be affine on $I_i$. The proof of Theorem \ref{MAINRES} is complete. \hfill $\Box$
\bigskip

\noindent
{\sl Proof of Proposition \ref{ALLLAMB}.} Assume first that $T(0+0)=0$. By assumption, there exists $a>0$ such that $T$ is continuous and expanding on $(0,a)$, $T:(0,a)\mapsto (0,T(a))$, and $T(x)>x$ for $x\in (0,a)$. Therefore, any $x\in (0,T(a))$ has a pre-image $y\in T^{-1}(x)$ such that $y<x$. 

By transitivity, the periodic orbits of $T$ are dense in $[0,1]$. Given $p>1$, let $(x_i)_{i=1}^p$ be one such orbit for which $x_1=\min_i x_i\in (0,T(a))$. Labelling in $(x_i)_{i=1}^p$ has been chosen so that $x_{i+1}=T(x_i)$ for $i=1,\cdots ,p-1$ and $T(x_p)=x_1$.
Let $x_0<x_1$ be the pre-image of $x_1$ in $(0,a)$.

Let $n\in\N$ and consider the pre-images $x_{0,i,n}\in T^{-n}(x_0),x_{1,i,n}\in T^{-n}(x_1)$ and $x_{p,i,n}\in T^{-n}(x_p)$ by the same affine branch of $T^n$. Transitivity and the fact that $T$ has a full branch ensure their existence for every branch of $T^n$, provided that $n$ is large enough.

For $k=1,\cdots ,n-1$, the iterates $T^{k+1}(x_{0,i,n}),T^{k+1}(x_{1,i,n}),T^{k+1}(x_{p,i,n})$ lie in the same atoms. Hence, there exist $c_{i,n}\in \R$ such that we have
\begin{align*}
\phi_\gamma(x_{0,i,n})-\phi_\gamma(x_{1,i,n})&=\sum_{k=0}^{n-1}\gamma^k\left(T^{k+1}(x_{0,i,n})-T^{k+1}(x_{1,i,n})\right) +\gamma^n\sum_{k=0}^{+\infty}\gamma^k\left(T^{k+1}(x_{0})-T^{k+1}(x_{1})\right) \\
&=c_{i,n}(x_{0,i,n} - x_{1,i,n})+\gamma^{n}(\phi_\gamma(x_{0})-\phi_\gamma(x_{1})),
\end{align*}
and $T(x_0)=x_1$ implies $\phi_\gamma(x_{0})-\phi_\gamma(x_{1})=x_1-(1-\gamma)\phi_\gamma(x_1)$. Similarly, we have
\[ 
\phi_\gamma(x_{p,i,n})-\phi_\gamma(x_{1,i,n})=c_{i,n}(x_{p,i,n} - x_{1,i,n})+\gamma^{n}(x_1-(1-\gamma)\phi_\gamma(x_1)).
\]
Since $x_0<x_1<x_p$, the point $x_{1,i,n}$ must lie inside the interval delimited by $x_{0,i,n}$ and $x_{p,i,n}$. Hence, the first terms in these expressions of $\phi_\gamma(x_{0,i,n})-\phi_\gamma(x_{1,i,n})$ and $\phi_\gamma(x_{p,i,n})-\phi_\gamma(x_{1,i,n})$ must be of opposite sign. This implies that
\[
\max\left\{|\phi_\gamma(x_{0,i,n})-\phi_\gamma(x_{1,i,n})|,|\phi_\gamma(x_{p,i,n})-\phi_\gamma(x_{1,i,n})|\right\}\geq \gamma^n|x_1-(1-\gamma)\phi_\gamma(x_1)|.
\]
Now, there are $N_n\simeq e^{nh_\text{top}(T)}$ triplets $x_{0,i,n},x_{1,i,n},x_{p,i,n}$, one for each cylinder of length $n+1$, i.e.\ each element of $\bigcap_{k=0}^{n}T^{-k}([0,1))$. Therefore, we have 
\[
\text{Var}\phi_\gamma\big{|}_{[0,1]}\geq N_n\gamma^n|x_1-(1-\gamma)\phi_\gamma(x_1)|,
\]
and to ensure the claim, it suffices to show that $x_1\neq (1-\gamma)\phi_\gamma(x_1)$.

Using that $(x_i)_{i=1}^p$ is $p$-periodic, one gets
\[
\phi_\gamma(x_1)=\frac{x_2+\gamma x_3+\cdots +\gamma^{p-1}x_p+\gamma^px_1}{1-\gamma^p}\,
\]
and the assumptions $x_1=\min_i x_i$ and $p>1$ imply
\[
x_2+\gamma x_3+\cdots +\gamma^{p-1}x_p+\gamma^px_1>\sum_{k=0}^p\gamma^k x_1
\]
from where it follows that $(1-\gamma)\phi_\gamma(x_1)>x_1$. This concludes the proof in the case $T(0+0)=0$. The case $T(1-0)=1$ can be treated identically.

Finally, in the case where there exist $\underline{a}<a<\overline{a}$ such that $T(\underline{a})=T(\overline{a})=a$ and either $T(a)=0$ and $T(0)>0$ or $T(a)=1$ and $T(1)<1$, a reasoning similar to the previous proof can be applied. The final condition for infinite variation becomes $0-(1-\gamma)\phi_\gamma(0)\neq 0$ (resp.\ $1-(1-\gamma)\phi_\gamma(1)\neq 0$), which certainly holds because $\phi_\gamma(0)\geq T(0)>0$ (resp. $\phi_\gamma(1)\leq T(1)+\frac{\gamma}{1-\gamma}<1+\frac{\gamma}{1-\gamma}$).\hfill $\Box$

\end{document}